\documentclass[12pt]{amsart}
\usepackage{amsfonts,amssymb,amscd,amsmath,enumerate,verbatim,color}
\usepackage[latin1]{inputenc}
\usepackage{amscd}
\usepackage{latexsym}

\usepackage[dvipdfmx]{graphicx}
\usepackage{mathptmx}

\def\NZQ{\mathbb}               

\def\ZZ{{\NZQ Z}}
\def\RR{{\NZQ R}}

\def\frk{\mathfrak}               

\def\Phi{{\frk N}}
\def\ab{{\mathbf a}}

\def\cb{{\mathbf c}}
\def\eb{{\mathbf e}}
\def\tb{{\mathbf t}}

\def\wb{{\mathbf w}}
\def\xb{{\mathbf x}}
\def\yb{{\mathbf y}}

\def\opn#1#2{\def#1{\operatorname{#2}}} 
\opn\gr{gr}

\def\Hc{{\mathcal H}}

\def\Gc{{\mathcal G}}

\def\Pc{{\mathcal P}}
\def\Qc{{\mathcal Q}}

\newtheorem{Theorem}{Theorem}[section]
\newtheorem{Lemma}[Theorem]{Lemma}
\newtheorem{Corollary}[Theorem]{Corollary}
\newtheorem{Proposition}[Theorem]{Proposition}

\theoremstyle{definition}
\newtheorem{Remark}[Theorem]{Remark}

\newtheorem{Example}[Theorem]{Example}

\newtheorem*{acknowledgement}{Acknowledgment}
\let\epsilon\varepsilon
\let\phi=\varphi
\let\kappa=\varkappa
\opn\dis{dis}
\opn\height{height}
\opn\dist{dist}
\def\pnt{{\raise0.5mm\hbox{\large\bf.}}}

\opn\Lex{Lex}
\opn\conv{conv}

\begin{document}

\title{Nef-partitions arising from unimodular configurations}
\author{Hidefumi Ohsugi and Akiyoshi Tsuchiya}
\address{Hidefumi Ohsugi,
	Department of Mathematical Sciences,
	School of Science and Technology,
	Kwansei Gakuin University,
	Sanda, Hyogo 669-1337, Japan} 
\email{ohsugi@kwansei.ac.jp}

\address{Akiyoshi Tsuchiya,
Graduate school of Mathematical Sciences,
University of Tokyo,
Komaba, Meguro-ku, Tokyo 153-8914, Japan} 
\email{akiyoshi@ms.u-tokyo.ac.jp}


\begin{abstract}
Reflexive polytopes have been studied from viewpoints of combinatorics, commutative algebra and algebraic geometry.
A nef-partition of a reflexive polytope $\mathcal{P}$ is a decomposition $\mathcal{P}=\mathcal{P}_1+\cdots+\mathcal{P}_r$ such that each $\mathcal{P}_i$ is a lattice polytope containing the origin.
Batyrev and van Straten gave a combinatorial method for explicit constructions of mirror pairs of Calabi-Yau complete intersections obtained from nef-partitions.
In the present paper, by means of Gr\"{o}bner basis techniques, we give a large family of nef-partitions arising from unimodular configurations.
\end{abstract}

\maketitle

\section*{Introduction}
A \textit{lattice polytope} is a convex polytope all of whose vertices have integer coordinates.
A lattice polytope $\Pc \subset \RR^d$ of dimension $d$ is 
called {\em reflexive} 
if the origin ${\bf 0}$ of $\RR^d$ belongs to the interior of $\Pc$ and 
if the dual polytope 
\[
\Pc^{\vee} = \{ {\bf x} \in \RR^{d} \, : \, \langle {\bf x}, {\bf y} \rangle \le 1,
\, \forall {\bf y} \in \Pc \}
\]
is again a lattice polytope.  Here $\langle {\bf x}, {\bf y} \rangle$
is the canonical inner product of $\RR^d$.
It is known that reflexive polytopes correspond to Gorenstein toric Fano varieties, and they are related to
mirror symmetry (see, e.g., \cite{mirror,Cox}).
Combinatorial aspects of mirror symmetry motivate the following definition:
a \textit{nef-partition of length $r$} is a decomposition $\Pc=\Pc_1+\cdots+\Pc_r$ of a $d$-dimensional reflexive polytope $\Pc \subset \RR^d$ into a Minkowski sum of $r$ lattice polytopes $\Pc_1,\ldots,\Pc_r$ such that each ${\bf 0} \in \Pc_i$.
Nef-partitions give many explicit constructions of mirrors of Calabi-Yau complete intersections (\cite{BN08,BS95,Bo93}).

In the present paper, 
making use of algebraic techniques involving Gr\"{o}bner bases, we give a large family of nef-partitions arising from unimodular configurations.
Given positive integers $d$ and $n$, let $\ZZ^{d \times n}$
denote the set of all $d \times n$ integer matrices.
A \textit{configuration} of $\RR^d$ is a matrix $A \in \ZZ^{d \times n}$, for which there exists an affine hyperplane $\Hc \subset \RR^d$ not passing ${\bf 0}$ such that each column vector of $A$ lies $\Hc$.
An integer matrix $A \in \ZZ^{d \times n}$ of rank $d$ is called \textit{unimodular} if all nonzero maximal minors of $A$ have the same absolute value. 
Unimodular matrices are important in polyhedral combinatorics and combinatorial optimization \cite{integer}. 
Given an integer matrix $A=(\ab_1,\ldots,\ab_n) \in \ZZ^{d \times n}$, let $\Pc_{A}:=\conv(\{\ab_1,\ldots, \ab_n\})$.
A famous example of a unimodular matrix is the \textit{incidence matrix} $A_G$ of a bipartite graph $G$ (by deleting redundant rows) and $\Pc_{A_G}$ is called the \textit{edge polytope} of $G$ (see Section~\ref{epsection}).
 
In order to give a family of nef-partitions, we consider whether the Cayley sum of given lattice polytopes is Gorenstein. A lattice polytope $\Pc \subset \RR^d$ is called {\em Gorenstein} of index $r$ if $r\Pc$ is unimodularly equivalent to a reflexive polytope.
In particular, a reflexive polytope is Gorenstein of index $1$.
Gorenstein polytopes are of interest in combinatorial commutative algebra, mirror symmetry, and tropical geometry (we refer to \cite{BJ,BN08,Joswig}). 
Given $r$ lattice polytopes $\Pc_1,\ldots, \Pc_r \subset \RR^d$, the \textit{Cayley sum} of $\Pc_1,\ldots, \Pc_r$ is the lattice polytope
\[
\Pc_1*\cdots*\Pc_r:=\conv(\{ \{\eb_1\} \times \Pc_1, \ldots, \{\eb_{r-1}\} \times \Pc_{r-1}, \{{\bf 0}\} \times \Pc_r\}) \subset \RR^{r-1} \times \RR^{d},
\]
where $\eb_1,\ldots,\eb_{r-1}$ are the standard basis of $\RR^{r-1}$.
Then it is known that $\Pc_1+\cdots+\Pc_r$ is Gorenstein of index $1$ if (and only if) $\Pc_1*\cdots*\Pc_r$ is Gorenstein of index $r$ (\cite[Theorem 2.6]{BN08}). 
On the other hand, our interest is in whether $\Pc_1+\cdots+\Pc_r$ possesses a regular unimodular triangulation. 
A \textit{unimodular simplex} is a lattice simplex which is unimodularly equivalent to the standard simplex $\Delta^d \subset \RR^d$, the convex hull of ${\bf 0}$ together with $\eb_1,\ldots,\eb_d$.
Equivalently, a full-dimensional lattice simplex in $\RR^d$ is unimodular if and only if it has the minimal possible Euclidean volume, $1/d!$.
A triangulation of a lattice polytope is called \textit{unimodular} if every maximal simplex is unimodular.
A full-dimensional lattice polytope $\Pc \subset \RR^d$  is called \textit{spanning} if it holds $\sum_{\ab \in \Pc \cap \ZZ^d}\ZZ(\ab,1)=\ZZ^{d+1}$.
In general, we say that a lattice polytope is spanning if it is unimodularly equivalent to a full-dimensional spanning polytope.
Sturmfels gave a one-to-one correspondence between a regular triangulation of a lattice polytope and the radical of an initial ideal of its toric ideal (\cite[Section 8]{sturmfels1996}).
In particular, for a spanning polytope, its regular triangulation is unimodular if and only if the associated initial ideal is squarefree.
Moreover, it follows that if $\Pc_1*\cdots *\Pc_r$ possesses a regular unimodular triangulation for spanning polytopes $\Pc_1,\ldots, \Pc_r \subset \RR^d$ of dimension $d$, then so does $\Pc_1+\cdots+\Pc_r$ (\cite[Theorem 2.2]{HOTCayley}). In particular, from \cite[Theorem 0.4]{TCayley} one has
\begin{equation}
\label{eq:oda}
(\Pc_1 \cap \ZZ^d)+\cdots + (\Pc_r \cap \ZZ^d)=(\Pc_1+\cdots+\Pc_r) \cap \ZZ^d.	
\end{equation}
In \cite{Oda}, Oda asked when the equation (\ref{eq:oda}) holds.
Recently, this question and the following conjecture
are among the current trends of the research
on lattice polytopes.

\bigskip

\noindent
{\bf Oda Conjecture.}
Every smooth polytope has the integer decomposition property.

\bigskip

For a lattice polytope $\Pc$, we let $-\Pc:=\{-\ab : \ab \in \Pc \}$. The main results of the present paper are the following two theorems.
\begin{Theorem}
\label{main1}
	Let $A=(\ab_1,\ldots,\ab_n) \in \ZZ^{d \times n}$ be a unimodular configuration.
	Assume that $\Pc_A \cap \ZZ^d = \{\ab_1,\ldots,\ab_n \}$ and $\Pc_A$ is spanning. 
	Then
	\begin{enumerate}
		\item[{\rm (1)}] $\Pc_A*(-\Pc_A)$ is Gorenstein of index $2$ with a regular unimodular triangulation{\rm ;}
		\item[{\rm (2)}] $\Pc_A+(-\Pc_A)$ is reflexive with a regular unimodular triangulation. In particular, $(\Pc_A-\ab)+(-\Pc_A+\ab)$ is a nef-partition, where $\ab$ is an arbitrary lattice point in $\Pc_A${\rm ;}
		\item[{\rm (3)}] One has $(\Pc_A \cap \ZZ^d)+(-\Pc_A \cap \ZZ^d)=(\Pc_A+(-\Pc_A)) \cap \ZZ^d$.
	\end{enumerate}
\end{Theorem}

\begin{Theorem}
\label{main2}
	Let $A=(\ab_1,\ldots,\ab_n) \in \ZZ^{d \times n}$ be a unimodular configuration and set $A_0=(A, {\bf 0}) \in \ZZ^{d \times (n+1)}$.
	Assume that $\Pc_{A_0} \cap \ZZ^d=\{\ab_1,\ldots,\ab_n,{\bf 0}\}$ and $\Pc_{A_0}$ is spanning. 
	Then
	\begin{enumerate}
		\item[{\rm (1)}] $\Pc_{A_0}*(-\Pc_{A_0})$ is Gorenstein of index $2$ with a regular unimodular triangulation{\rm ;}
		\item[{\rm (2)}] $\Pc_{A_0}+(-\Pc_{A_0})$ is reflexive with a regular unimodular triangulation. In particular, $\Pc_{A_0}+(-\Pc_{A_0})$ is a nef-partition{\rm ;}
		\item[{\rm (3)}] One has $(\Pc_{A_0} \cap \ZZ^d)+(-\Pc_{A_0} \cap \ZZ^d)=(\Pc_{A_0}+(-\Pc_{A_0})) \cap \ZZ^d$.
	\end{enumerate}
\end{Theorem}
\begin{Remark}
	In general, for a lattice polytope $\Pc \subset \RR^d$, if $\Pc+(-\Pc)$ is reflexive, then $(\Pc-\ab)+(-\Pc+\ab)$ is a nef-partition, where $\ab$ is an arbitrary lattice point in $\Pc$.
\end{Remark}

In Section \ref{epsection}, we will apply these theorems to unimodular configurations arising from finite simple graphs. In fact, we will show
that for any finite simple graph $G$ all pairs of whose odd cycles have a common vertex, $(\Pc_{A_G}-\ab)+(-\Pc_{A_G}+\ab)$ is a nef-partition, where $\ab$ is an arbitrary lattice point in $\Pc_{A_G}$ (Theorem~\ref{edge1}). Note that for such a graph, the edge polytope is unimodular, i.e., all triangulations are unimodular.
Moreover, we will show that for any finite bipartite graphs $G$, $\Pc_{(A_G)_0}+(-\Pc_{(A_G)_0})$ is a nef-partition (Theorem~\ref{edge2}). 

The present paper is organized as follows:
In Section~\ref{gbsection}, we will investigate three types of toric ideals arising from unimodular configurations. In particular, they possess squarefree initial ideals with respect to some reverse lexicographic orders (Theorems \ref{pmGB}, \ref{CayleyGB} and \ref{AzeroGB}).
In Section~\ref{proofsection}, 
we compare $h^*$-polynomials with $h$-polynomials for the
ideals and initial ideals considered in Theorems  \ref{pmGB}, \ref{CayleyGB} and \ref{AzeroGB},
thus proving Theorems \ref{main1} and \ref{main2}.
Finally, in Section \ref{epsection}, we apply these theorems to unimodular configurations arising from finite simple graphs.

\section{Reverse lexicographic Gr\"{o}bner bases of unimodular configurations}

\label{gbsection}

Let  $A =(\ab_1,\ldots,\ab_n)  \in \ZZ^{d \times n}$ be a configuration.
From now on, we always assume that $A$ has no repeated columns.
Let $K[\tb^{\pm}]=K[t_1^{\pm 1},\dots,t_d^{\pm 1}]$ be a
Laurent polynomial ring over a field $K$.
Given an integer vector $\alpha = (\alpha_1,\dots,\alpha_d) \in \ZZ^d$,
let $\tb^\alpha = t_1^{\alpha_1} \dots t_d^{\alpha_d}$
be a Laurent monomial in $K[\tb^\pm]$.
The {\em toric ring} $K[A] \subset K[\tb^\pm]$ of $A$ is a semigroup ring generated by 
$\tb^{\ab_1},\dots,\tb^{\ab_n}$ over $K$.
Let $K[\xb] = K[x_1,\dots,x_n]$ be a polynomial ring over $K$ with each ${\rm deg}(x_i)=1$.
Then the {\em toric ideal} $I_A$ of $A$ is the kernel of 
a surjective ring homomorphism 
$
\varphi_{A} : K[\xb] \rightarrow K[A]
$
defined by 
$\varphi_{A} (x_i) = \tb^{\ab_i}$ for each $1 \le i \le n$.
It is known that $I_A$ is generated by homogeneous binomials (of degree $\ge 2$)
if $I_A \neq \{0\}$.
In addition, any reduced  Gr\"{o}bner basis of $I_A$ consists of 
homogeneous binomials.
See \cite[Section~3]{BinomialIdeals} for the introduction to toric rings and ideals.
For a lattice polytope $\Pc \subset \RR^d$ with $\Pc \cap \ZZ^d =\{ \ab_1,\ldots, \ab_n\}$, we define a configuration 
\[A_{\Pc} :=
\left(\begin{array}{ccc}
\ab_1 & \cdots  & \ab_n\\
1  &  \cdots & 1
\end{array}
\right)
\in \ZZ^{(d+1) \times n}.\]
Then the toric ring $K[\Pc]$ of $\Pc$ is $K[A_{\Pc}]$ and the toric ideal $I_{\Pc}$ of $\Pc$ is $I_{A_{\Pc}}$.

Given an integer matrix $A =(\ab_1,\ldots,\ab_n)  \in \ZZ^{d \times n}$,
 let $A^\pm$ be a configuration 
\[A^\pm :=
\left(\begin{array}{ccccccc}
\ab_1 & \cdots & \ab_n & -\ab_1 & \cdots & -\ab_n & {\bf 0}\\
1 & \cdots & 1& 1 & \cdots & 1 & 1
\end{array}
\right)
\in \ZZ^{(d+1) \times (2n+1)}.\]
The configuration $A^\pm$ is called the {\em centrally symmetric configuration} of $A$.
The toric ideal $I_{A^\pm}$ of a configuration $A^\pm$ is the kernel of 
a ring homomorphism 
\[
\varphi_{A^\pm} : K[x_1,\dots,x_n,y_1,\dots,y_n,z] \rightarrow K[t_1^{\pm 1},\dots, t_d^{\pm 1}, s]
\]
defined by 
$\varphi_{A^\pm} (x_i) = \tb^{\ab_i} s$,
$\varphi_{A^\pm} (y_i) = \tb^{-\ab_i} s$
for each $1 \le i \le n$,
and
$\varphi_{A^\pm} (z) =s$.
The following proposition is given in \cite[Theorems~2.7, 2.15 and Corollary 2.8]{OHcentrally}.

\begin{Proposition}
\label{OHresults}
	Let $A \in \ZZ^{d \times n}$ be a unimodular matrix.
Then we have the following{\rm :}
\begin{itemize}
\item[{\rm (a)}] 
The initial ideal of $I_{A^\pm}$ is squarefree
with respect to a reverse lexicographic order such that
the smallest variable is $z${\rm ;}
\item[{\rm (b)}]
The toric ring $K[A^\pm]$ is normal and Gorenstein.
In particular the $h$-polynomial $h(K[A^\pm],t)$ of $K[A^\pm]$ is 
a palindromic polynomial of degree $d$, i.e.,
$h(K[A^\pm],t)=t^d h(K[A^\pm],t^{-1})$.
\end{itemize}
\end{Proposition}

In \cite{OHcentrally},
Proposition \ref{OHresults} (a) is shown by studying 
the corresponding triangulation of $\Pc_{A^\pm}$,
and Proposition \ref{OHresults} (b) is shown by using
Proposition \ref{OHresults} (a).
We now study
the reverse lexicographic Gr\"{o}bner bases of $I_{A^\pm}$
in detail when $A$ is a unimodular {\em configuration}.

\begin{Theorem}
\label{pmGB}
	Let $A \in \ZZ^{d \times n}$ be a unimodular configuration.
Let $<$ be the reverse lexicographic order induced by
the ordering of variables $z < x_1 < y_1 < \cdots <x_n < y_n$.
Then the reduced Gr\"{o}bner basis $\Gc$ of $I_{A^\pm}$
with respect to $<$ is of the form
\[
\{x_i y_i - z^2 : i \in [n]\} \cup \{g_1,\dots, g_s\},
\]
where $[n]:=\{1,\ldots,n\}$, $g_i \in K[x_1,\dots,x_n,y_1,\dots,y_n]$, ${\rm in}_<(g_i) \in  K[x_2,\dots,x_n,y_2,\dots,y_n]$, and any monomial of $g_i$
is squarefree.
\end{Theorem}

\begin{proof}
Let $\Gc$ be the reduced Gr\"{o}bner basis
and let $\Gc_1=  \{x_i y_i - z^2 : i \in [n]\}$.
It is easy to see that $\Gc_1$ is a subset of $\Gc$. 
Let 
\begin{eqnarray}
\label{csbinomial}
g= p -q=
x_1^{u_1} \dots x_n^{u_n} 
y_1^{v_1} \dots y_n^{v_n}
-
z^\ell 
x_1^{u_1'} \dots x_n^{u_n'} 
y_1^{v_1'} \dots y_n^{v_n'} 
\end{eqnarray}
be a binomial in $\Gc \setminus \Gc_1$ with ${\rm in}_<(g) = p$.
Since $g$ belongs to the reduced Gr\"{o}bner basis,
(i) $g$ must be irreducible, (ii) $p$ belongs to the minimal set of monomial generators
of ${\rm in}_<(I_{A^\pm})$, and (iii) $q \notin {\rm in}_<(I_{A^\pm})$.
By Proposition~\ref{OHresults} (a), $p$ is squarefree, that is,
$u_i, v_i \in \{0,1\}$.
Moreover, 
if $u_i = v_i = 1$, then $p$ is divisible by the initial monomial $x_i y_i$ of $x_i y_i -z^2$, a contradiction to the assumption that $\Gc$ is the reduced Gr\"{o}bner basis.
Hence $(u_i ,v_i) \in \{(0,0), (0,1), (1,0)\}$ for each $i$.
Since $g$ belongs to $I_{A^\pm}$, we have 
\begin{equation}
\label{basiceq}
\sum_{k=1}^n u_k \ab_k + 
\sum_{k=1}^n v_k (-\ab_k) =
\sum_{k=1}^n u_k' \ab_k + 
\sum_{k=1}^n v_k' (-\ab_k)
\end{equation}
and 
$
\sum_{k=1}^n u_k + \sum_{k=1}^n v_k
=
\ell + \sum_{k=1}^n u_k' + \sum_{k=1}^n v_k'
$.
Moreover since $A$ is a configuration, there exists a vector $\cb \in \RR^d$ such that
the inner product $\left< \ab_i , \cb \right>= 1 $ for all $1 \le i \le n$.
Taking the inner product of equation (\ref{basiceq}) with $\cb$,
we have
\[
\sum_{k=1}^n u_k  -
\sum_{k=1}^n v_k  =
\sum_{k=1}^n u_k'  -
\sum_{k=1}^n v_k'  .
\]
Thus 
$\ell = 2 (\sum_{k=1}^n u_k  - \sum_{k=1}^n u_k')$, and hence $\ell$ is even.

Suppose that $\ell > 0$.
Since $\ell$ is even, we have $\ell \ge 2$.
Let $k = \min\{ i \in [n] : u_i \mbox{ or } v_i \mbox{ is } 1\}$.
If $u_k =1$, then $h= p/x_k -  q y_k/z^2$ belongs to $I_{A^\pm}$.
Since $v_k = 0$, we have $h \neq 0$.
Thus ${\rm in}_<(h) =p/x_k$ divides $p$, a contradiction.
If $v_k =1$, then $h'= p/y_k -  q x_k/z^2$ belongs to $I_{A^\pm}$.
Since $u_k = 0$, we have $h' \neq 0$.
Thus ${\rm in}_<(h') =p/y_k$ divides $p$, a contradiction.
Therefore $\ell = 0$, and hence $g \in  K[x_1,\dots,x_n,y_1,\dots,y_n]$.

Since $g$ is irreducible and $<$ is a reverse lexicographic order,
 $x_1$ does not appear in $p$.
Suppose that $y_1$ appears in $p$.
Since $<$ is a reverse lexicographic order, $q$ is divisible by $x_1$.
Then $h= p z^2/y_1 -  q x_1$ belongs to $I_{A^\pm}$.
The initial monomial of $h$ is $q x_1$ that is not squarefree.
Since the initial ideal is squarefree, $q$ belongs to the initial ideal.
This contradicts that $g$ belongs to the reduced  Gr\"{o}bner basis.
Thus 
$p \in  K[x_2,\dots,x_n,y_2,\dots,y_n]$.

Finally, suppose that $q$ is not squarefree.
Let $g = \xb^{{\mu}_1} \yb^{\eta_1} - \xb^{\mu_2} \yb^{\eta_2}$.
By equation (\ref{basiceq}), $\xb^{\mu_1 + \eta_2} - \xb^{\mu_2 + \eta_1}$ belongs to
the toric ideal $I_{A}$
and $\xb^{\mu_1 + \eta_2} - \xb^{\mu_2 + \eta_1}$
has a monomial that is not squarefree.
An irreducible binomial $h$ of $I_A$ is called a {\em circuit} of $I_A$ if
there is no binomial $0 \ne h' \in I_A$ such that ${\rm supp}(h') \subsetneq {\rm supp}(h) $.
Here ${\rm supp}(h)$ is the set of all variables appearing in $h$.
By \cite[Lemma~4.32]{BinomialIdeals} there exists a circuit $f= \xb^{\wb_1+\wb_1'} - \xb^{\wb_2+\wb_2'}\in I_A$
such that 
\[
{\rm supp}( \xb^{\wb_1} )  \subset {\rm supp}( \xb^{\mu_1} ) ,\ 
{\rm supp}( \xb^{\wb_1'} )  \subset {\rm supp}( \xb^{\eta_2} ) ,\ 
{\rm supp}( \xb^{\wb_2} )  \subset {\rm supp}( \xb^{\mu_2} ) ,\ 
{\rm supp}( \xb^{\wb_2'} )  \subset {\rm supp}( \xb^{\eta_1} ) .
\]
Since $A$ is unimodular, by \cite[Theorem 4.35]{BinomialIdeals}, both $\xb^{\wb_1+\wb_1'}$ and $\xb^{\wb_2+\wb_2'}$ are squarefree.
%
%
%
Hence we have $\deg (f) < \deg (g)$.
Note that $f'=\xb^{\wb_1} \yb^{\wb_2'} z^\alpha - \xb^{\wb_2} \yb^{\wb_1'} z^\beta$ belongs to 
$I_{A^\pm}$ for some $\alpha, \beta \in \ZZ_{\ge 0}$ with $\alpha \beta=0$.
If $\alpha >0$, then $ \xb^{\wb_2} \yb^{\wb_1'}$ belongs to the initial ideal,
and hence so does $q$.
This contradicts that $g$ belongs to the reduced  Gr\"{o}bner basis.
If $\beta >0$, then $\xb^{\mu_1} \yb^{\eta_1}$ is divisible by 
${\rm in}_<(f') = \xb^{\wb_1} \yb^{\wb_2'}$
and hence $\xb^{\mu_1} \yb^{\eta_1} = {\rm in}_<(f') $.
Then $q$ is the initial monomial of a binomial $f' -g \in I_{A^\pm}$, a contradiction.
If $\alpha=\beta=0$, then $\deg (f') < \deg (g)$ and 
either $p$ or $q$ is divisible by ${\rm in}_<(f')$,
a contradiction.
Thus it follows that $q$ is squarefree.
\end{proof}

	Let $A=(\ab_1,\ldots,\ab_n) \in \ZZ^{d \times n}$ be a unimodular configuration.
	Assume that $\Pc_A \cap \ZZ^d = \{\ab_1,\ldots,\ab_n \}$. 
The toric ideal $I_{\Pc_A*(-\Pc_A)}$ of $\Pc_A*(-\Pc_A)$ is the kernel of 
a ring homomorphism 
\[
\varphi_* : K[x_1,\dots,x_n,y_1,\dots,y_n] \rightarrow K[t_0^{\pm 1}, t_1^{\pm 1},\dots, t_d^{\pm 1}, s]
\]
defined by 
$\varphi_* (x_i) = t_0 \tb^{\ab_i} s$,
$\varphi_* (y_i) = \tb^{-\ab_i} s$
for each $1 \le i \le n$.

\begin{Theorem}
\label{CayleyGB}
Work with the same notation as above.
Let $<'$ be the reverse lexicographic order induced by
the ordering of variables $x_1 <' y_1 <' \cdots <'x_n <' y_n$.
Then the reduced Gr\"{o}bner basis of $I_{\Pc_A*(-\Pc_A)}$
with respect to $<'$ is
\[
\Gc' = \{x_i y_i - x_1 y_1 : i =2,\dots, n\} \cup \{g_1,\dots, g_s\},
\]
where $ \{g_1,\dots, g_s\}$ is such that $\{x_i y_i - z^2 : i \in [n]\} \cup \{g_1,\dots, g_s\}$ is
the reduced Gr\"{o}bner basis of $I_{A^\pm}$ from Theorem~\ref{pmGB}.
\end{Theorem}

\begin{proof}
It is easy to see that $x_i y_i - x_1 y_1$ belongs to $I_{\Pc_A*(-\Pc_A)}$.
In the proof of Theorem~\ref{pmGB}, it is proved that each $g_i$ is of the form (\ref{csbinomial})
with $0 = \ell = 2(\sum_{k=1}^n u_k - \sum_{k=1}^n u_k')$.
Thus 
$\sum_{k=1}^n u_k = \sum_{k=1}^n u_k'$.
Hence each $g_i$ belongs to $I_{\Pc_A*(-\Pc_A)}$
and $\Gc'$ is a subset of  $I_{\Pc_A*(-\Pc_A)}$.
The initial monomial of each binomial in $\Gc'$ is 
${\rm in}_{<'} (x_i y_i -x_1 y_1) =x_i y_i $
and
${\rm in}_{<'} (g_i) = {\rm in}_< (g_i)$.

Suppose that $\Gc'$ is not a Gr\"{o}bner basis of $I_{\Pc_A*(-\Pc_A)}$.
By \cite[Theorem~3.11]{BinomialIdeals}, there exists an irreducible binomial $0 \ne f= u - v \in I_{\Pc_A*(-\Pc_A)}$
such that $u,v \notin \left<  {\rm in}_{<'} (\Gc')  \right>$.
Then $f \in I_{\Pc_{A^\pm}} \cap  K[x_1,\dots,x_n,y_1,\dots,y_n] $.
Since $\Gc$ in Theorem \ref{pmGB} is a Gr\"{o}bner basis of $ I_{\Pc_{A^\pm}}$,
there exists $g \in \Gc$ such that ${\rm in}_<(g)$
divides ${\rm in}_<(f) = {\rm in}_{<'} (f)$.
If $g \neq x_1 y_1 -z^2$, then there exists $g' \in \Gc'$ such that ${\rm in}_<(g) = {\rm in}_{<'}(g')$, a contradiction.
Hence $g = x_1 y_1 -z^2$ and ${\rm in}_{<'} (f)$
is divisible by $x_1 y_1$.
Since $f$ is irreducible, the monomial $f - {\rm in}_{<'} (f)$ does not contain $x_1$.
This contradicts the definition of $<'$.
Thus $\Gc'$ is a Gr\"{o}bner basis of $I_{\Pc_A*(-\Pc_A)}$.
Moreover, since $\{x_i y_i - z^2 : i \in [n]\} \cup \{g_1,\dots, g_s\}$ is
the reduced Gr\"{o}bner basis of $I_{A^\pm}$ and 
$x_1$ and $y_1$ do not appear in any ${\rm in}_{<'} (g_i)$,
it follows that $\Gc'$ is reduced .
\end{proof}

	Let $A=(\ab_1,\ldots,\ab_n) \in \ZZ^{d \times n}$ be a unimodular configuration and set $A_0=(A, {\bf 0}) \in \ZZ^{d \times (n+1)}$.
	Assume that $\Pc_{A_0} \cap \ZZ^d=\{\ab_1,\ldots,\ab_n,{\bf 0}\}$. 
The toric ideal $I_{\Pc_{A_0}*(-\Pc_{A_0})}$ of $\Pc_{A_0}*(-\Pc_{A_0})$ is the kernel of 
a ring homomorphism 
\[
\varphi_* : K[x_0, x_1,\dots,x_n,y_0,y_1,\dots,y_n] \rightarrow K[t_0^{\pm 1 }, t_1^{\pm 1},\dots, t_d^{\pm 1}, s]
\]
defined by 
$\varphi_* (x_i) = t_0 \tb^{\ab_i} s$,
$\varphi_* (y_i) = \tb^{-\ab_i} s$
for each $1 \le i \le n$,
and
$\varphi_* (x_0) =t_0 s$,
$\varphi_* (y_0) =s$.

\begin{Theorem}
\label{AzeroGB}
Work with the same notation as above.
Let $<'$ be the reverse lexicographic order induced by
the ordering of variables $x_0 <' y_0 <' \cdots <'x_n <' y_n$.
Then the reduced Gr\"{o}bner basis of $I_{\Pc_{A_0}*(-\Pc_{A_0})}$
with respect to $<'$ is
\[
\Gc' = \{x_i y_i - x_0 y_0 : i  \in [n]\} \cup \{g_1,\dots, g_s\},
\]
where $ \{g_1,\dots, g_s\}$ is such that $\{x_i y_i - z^2 : i \in [n]\} \cup \{g_1,\dots, g_s\}$ is
the reduced Gr\"{o}bner basis of $I_{A^\pm}$ from Theorem~\ref{pmGB}.

\end{Theorem}

\begin{proof}
Each $g_i$ belongs to $I_{\Pc_{A_0}*(-\Pc_{A_0})}$
since $I_{\Pc_A*(-\Pc_A)} \subset I_{\Pc_{A_0}*(-\Pc_{A_0})}$.
Since each $x_i y_i - x_0 y_0$ belongs to $I_{\Pc_{A_0}*(-\Pc_{A_0})}$, the set $\Gc'$ is a subset of $I_{\Pc_{A_0}*(-\Pc_{A_0})}$.
The initial monomial of each binomial in $\Gc'$ is 
${\rm in}_{<'} (x_i y_i -x_0 y_0) =x_i y_i $
and
${\rm in}_{<'} (g_i) = {\rm in}_< (g_i)$.
Note that $ {\rm in}_{<} (\Gc)  =  {\rm in}_{<'} (\Gc') $.

Suppose that $\Gc'$ is not a Gr\"{o}bner basis of $I_{\Pc_{A_0}*(-\Pc_{A_0})}$.
By \cite[Theorem~3.11]{BinomialIdeals}, there exists an irreducible binomial $0 \ne f= u - v \in I_{\Pc_{A_0}*(-\Pc_{A_0})}$
such that $u,v \notin \left<  {\rm in}_{<'} (\Gc')  \right>$.
Let
\[f=
x_0^{u_0}
x_1^{u_1} \dots x_n^{u_n} 
y_0^{v_0}
y_1^{v_1} \dots y_n^{v_n}
-
x_0^{u_0'}
x_1^{u_1'} \dots x_n^{u_n'} 
y_0^{v_0'}
y_1^{v_1'} \dots y_n^{v_n'}.
\]
Then we have 
\begin{eqnarray}
\label{jou}
\sum_{k=1}^n u_k \ab_k + 
\sum_{k=1}^n v_k (-\ab_k) =
\sum_{k=1}^n u_k' \ab_k + 
\sum_{k=1}^n v_k' (-\ab_k),
\end{eqnarray}
$
\sum_{k=0}^n u_k  
=
\sum_{k=0}^n u_k'  
$
and
$
 \sum_{k=0}^n v_k
=
\sum_{k=0}^n v_k'
$.
Moreover since $A$ is a configuration, there exists a vector $\cb \in \RR^d$ such that
the inner product $\left< \ab_i , \cb \right> = 1$ for all $1 \le i \le n$.
Taking the inner product of equation (\ref{jou}) with $\cb$, we have
\[
\sum_{k=1}^n u_k  -
\sum_{k=1}^n v_k  =
\sum_{k=1}^n u_k'  -
\sum_{k=1}^n v_k'  .
\]
Thus $v_0 - u_0 = v_0' - u_0' $.

If $u_0 = u_0'=v_0= v_0'= 0$, then $f \in K[x_1,\dots,x_n,y_1,\dots,y_n]$.
Then $f$ belongs to $I_{A^\pm}$.
This contradicts that $\Gc$ is a Gr\"{o}bner basis of $I_{A^\pm}$.
Hence at least one of $u_0$, $u_0'$, $v_0$, $v_0'$ is positive.
Since $f$ is irreducible, we may assume that $u_0 > 0$ and $u_0' = 0$.
Then $v_0 = u_0 + v_0' >0$
and hence $v_0'=0$.
Thus $u_0 = v_0 > 0$ and $u_0' = v_0' = 0$, that is,
\[
f=
(x_0 y_0)^{u_0}
x_1^{u_1} \dots x_n^{u_n} 
y_1^{v_1} \dots y_n^{v_n}
-
x_1^{u_1'} \dots x_n^{u_n'} 
y_1^{v_1'} \dots y_n^{v_n'}.
\]
It then follows that 
\[
f'=
z^{2u_0}
x_1^{u_1} \dots x_n^{u_n} 
y_1^{v_1} \dots y_n^{v_n}
-
x_1^{u_1'} \dots x_n^{u_n'} 
y_1^{v_1'} \dots y_n^{v_n'}
\]
belongs to $I_{A^\pm}$.
Since $\Gc$ is a Gr\"{o}bner basis of $I_{A^\pm}$,
there exists $g \in \Gc$ such that ${\rm in}_<(g)$ divides
${\rm in}_<(f') = x_1^{u_1'} \dots x_n^{u_n'} 
y_1^{v_1'} \dots y_n^{v_n'}$.
Since there exists $g' \in \Gc'$ such that ${\rm in}_<(g) ={\rm in}_{<'} (g')$, this is a contradiction.
\end{proof}

\section{Proofs of Theorems}

\label{proofsection}

In the present section, we give proofs of Theorems~\ref{main1} and \ref{main2}
by using the Gr\"{o}bner bases constructed in Section~\ref{gbsection}.
First, we recall the theory of the $h^*$-polynomials of lattice polytopes.
Two lattice polytopes $\Pc \subset \RR^d$ and $\Pc' \subset \RR^{d'}$ are said to be \emph{unimodularly equivalent} if there exists an affine map from the affine span ${\rm aff}(\Pc)$ of $\Pc$ to the affine span ${\rm aff}(\Pc')$ of $\Pc'$ that maps $\ZZ^d \cap {\rm aff}(\Pc)$ bijectively onto $\ZZ^{d'} \cap {\rm aff}(\Pc')$ and that maps $\Pc$ to $\Pc'$. 
Note that every lattice polytope is unimodularly equivalent to a full-dimensional one.
Let $\Pc \subset \RR^d$ be a lattice polytope of dimension $d$.
Given a positive integer $m$, we define
\[L_{\Pc}(m)=|m \Pc \cap \ZZ^d|,\]
where $m\Pc:=\{m \ab : \ab \in \Pc \}$.
Ehrhart \cite{Ehrhart} proved that $L_{\Pc}(m)$ is a polynomial in $m$ of degree $d$ with the constant term $1$.
We say that $L_{\Pc}(m)$ is the \textit{Ehrhart polynomial} of $\Pc$.
Clearly, if $\Pc$ and $\Qc$ are unimodularly equivalent, then one has $L_{\Pc}(m)=L_{\Qc}(m)$.
The generating function of the lattice point enumerator, i.e., the formal power series
\[\text{Ehr}_\Pc(t)=1+\sum\limits_{k=1}^{\infty}L_{\Pc}(k)t^k\]
is called the \textit{Ehrhart series} of $\Pc$.
It is well known that it can be expressed as a rational function of the form
\[\text{Ehr}_\Pc(t)=\frac{h^*(\Pc,t)}{(1-t)^{d+1}}.\]
Then $h^*(\Pc,t)$ is a polynomial in $t$ of degree at most $d$ with nonnegative integer coefficients (\cite{Stanleynonnegative}) and it
is called
the \textit{$h^*$-polynomial} (or the \textit{$\delta$-polynomial}) of $\Pc$. 
A characterization of Gorenstein polytopes in terms of their $h^*$-polynomials is known.
In fact, a lattice polytope $\Pc$ of dimension $d$ is Gorenstein of index $r$ if and only if its $h^*$-polynomial is of degree $d+1-r$ and palindromic, i.e., $h^*(\Pc,t)=t^{d+1-r}h^*(\Pc,t^{-1})$ (\cite{DeNegriHibi}).

We say that a lattice polytope $\Pc \subset \RR^d$ possesses the \textit{integer decomposition property} if for any $k \in \ZZ_{> 0}$ and for any $\xb \in k \Pc \cap \ZZ^d$, there exist $k$ lattice points $\xb_1,\ldots,\xb_k$ in $\Pc$ with $\xb=\xb_1+\cdots+\xb_k$, and we call a lattice polytope with the integer decomposition property an \textit{IDP} polytope.
   IDP polytopes turn up in many fields of mathematics such as algebraic geometry, where they correspond to projectively normal embeddings of toric varieties, and commutative algebra, where they correspond to standard graded Cohen-Macaulay domains (see \cite{BrunsGubeladze}).
Moreover, the integer decomposition property is particularly important in the theory and application of integer programing \cite[\S 22.10]{integer}.
Note that a lattice polytope with a unimodular triangulation is IDP, and an IDP polytope is spanning.
Moreover, a lattice polytope is IDP if and only if its $h^*$-polynomial coincides with the $h$-polynomial of its toric ring.

We turn to the review of indispensable lemmata for our proofs of Theorems \ref{main1} and \ref{main2}.

\begin{Lemma}[{\cite[Corollary 6.1.5]{Monomial}}]
\label{lem:initialh}
	Let $S$ be a polynomial ring and $I \subset S$ a graded ideal of $S$. Fix a monomial order $<$ on $S$.
	Then $S/I$ and $S/{\rm in}_{<}(I)$ have the same Hilbert function, in particular, they have the same $h$-polynomial.
\end{Lemma}

 Since a lattice polytope $\Pc$ is unimodularly equivalent to 
a lattice polytope $\Pc'$ if and only if $\Pc * (-\Pc)$
 is unimodularly equivalent to  $\Pc' * (-\Pc')$
and since any initial ideal given in  Section~\ref{gbsection} is squarefree,
we can use the following lemma.

\begin{Lemma}[{\cite[Theorem 2.2]{HOTCayley}}]
\label{initialcomplex}
Let $\Pc_1, \dots, \Pc_r \subset \RR^d$ be spanning lattice polytopes
of dimension $d$.
Suppose that the toric ideal of $\Pc_1 * \cdots * \Pc_r$ has 
a squarefree initial ideal.
Then the toric ideal of $\Pc_1 + \cdots + \Pc_r$ has 
a squarefree initial ideal, and both 
$\Pc_1 * \cdots * \Pc_r$ and $\Pc_1 + \cdots + \Pc_r$
have a regular unimodular triangulation and are IDP.
\end{Lemma}

Now, we prove Theorems \ref{main1} and \ref{main2}.

\begin{proof}[Proof of Theorem~\ref{main2}]
	Let $A=(\ab_1,\ldots,\ab_n) \in \ZZ^{d \times n}$ be a unimodular configuration and set $A_0=(A, {\bf 0}) \in \ZZ^{d \times (n+1)}$.
	Assume that $\Pc_{A_0} \cap \ZZ^d=\{\ab_1,\ldots,\ab_n,{\bf 0}\}$ and $\Pc_{A_0}$ is spanning. 
From Theorem~\ref{AzeroGB}, the toric ideal $I_{\Pc_{A_0}*(-\Pc_{A_0})}$
has a squarefree initial ideal.
Hence by Lemma~\ref{initialcomplex},
the toric ideal of $\Pc_{A_0} + (-\Pc_{A_0})$ has 
a squarefree initial ideal, and both 
$\Pc_{A_0}*(-\Pc_{A_0})$ and $\Pc_{A_0} + (-\Pc_{A_0})$
have a regular unimodular triangulation and are IDP.
Then \cite[Theorem~0.4]{TCayley} guarantees that $(\Pc_{A_0} \cap \ZZ^d)+(-\Pc_{A_0} \cap \ZZ^d)=(\Pc_{A_0}+(-\Pc_{A_0})) \cap \ZZ^d$.

On the other hand, the minimal set of
monomial generators of the initial ideals ${\rm in}_<(I_{A^\pm})$
and ${\rm in}_{<'} (I_{\Pc_{A_0}*(-\Pc_{A_0})})$ given in 
Theorems~\ref{pmGB} and \ref{AzeroGB} are the same.
Thus 
\begin{eqnarray*}
 &  & K[x_1,\dots,x_n,y_1,\dots,y_n,z]/{\rm in}_<(I_{A^\pm})\\
&\simeq&
K[x_1,\dots,x_n,y_1,\dots,y_n]/\left<x_1 y_1,\dots,x_ny_n,{\rm in}_<(g_1), \dots, {\rm in}_<(g_s)\right>[z]
\end{eqnarray*}
and
\begin{eqnarray*}
 &  & K[x_0, x_1,\dots,x_n,y_0,y_1,\dots,y_n]/{\rm in}_{<'} (I_{\Pc_{A_0}*(-\Pc_{A_0})})\\
&\simeq&
K[x_1,\dots,x_n,y_1,\dots,y_n]/\left<x_1 y_1,\dots,x_ny_n,{\rm in}_<(g_1), \dots, {\rm in}_<(g_s)\right>[x_0, y_0].
\end{eqnarray*}
Since $\Pc_{A_0}*(-\Pc_{A_0})$ is IDP, it follows from Lemma \ref{lem:initialh} that
\begin{eqnarray}
h(K[A^\pm], t) &=& h(K[x_1,\dots,x_n,y_1,\dots,y_n,z]/{\rm in}_<(I_{A^\pm}), t) \label{aaa} \\
&=& h( K[x_0, x_1,\dots,x_n,y_0,y_1,\dots,y_n]/{\rm in}_{<'} (I_{\Pc_{A_0}*(-\Pc_{A_0})}), t) \label{bbb}\\
&=& h(K[\Pc_{A_0}*(-\Pc_{A_0})], t) \label{ccc}\\
&=& h^*(\Pc_{A_0}*(-\Pc_{A_0}), t). \label{ddd}
\end{eqnarray}
Here (\ref{aaa}) and (\ref{ccc}) follow from Lemma \ref{lem:initialh},
(\ref{bbb}) follows from the above,
and (\ref{ddd}) holds since $\Pc_{A_0}*(-\Pc_{A_0})$ is IDP.
From Proposition~\ref{OHresults}, $h(K[A^\pm], t)$ is palindromic and of degree $d$,
and so is $h^*(\Pc_{A_0}*(-\Pc_{A_0}), t) $.
Since the dimension of  $\Pc_{A_0}*(-\Pc_{A_0})$ is $d+1$,
 $\Pc_{A_0}*(-\Pc_{A_0})$ is Gorenstein of index $2$.
By \cite[Theorem~2.6]{BN08}, $\Pc_{A_0}+ (-\Pc_{A_0})$ is reflexive.
\end{proof}

\begin{proof}[Proof of Theorem~\ref{main1}]
The proof is the same as the proof of Theorem~\ref{main2} except for 
the following discussion.
Comparing the initial ideals ${\rm in}_<(I_{A^\pm})$
and ${\rm in}_{<'} (I_{\Pc_A*(-\Pc_A)})$ given in 
Theorems~\ref{pmGB} and \ref{CayleyGB}, it follows that 
\begin{eqnarray*}
 & &K[x_1,\dots,x_n,y_1,\dots,y_n,z]/{\rm in}_<(I_{A^\pm})\\
&\simeq&
(K[x_1,y_1]/\left<x_1 y_1\right>)\\
& & \ \ \ \   \otimes_K
( K[x_2,\dots,x_n,y_2,\dots,y_n]/\left<x_2y_2,\dots,x_ny_n,{\rm in}_<(g_1), \dots, {\rm in}_<(g_s)\right>)[z]
\end{eqnarray*}
and
\begin{eqnarray*}
 & & K[x_1,\dots,x_n,y_1,\dots,y_n]/{\rm in}_{<'} (I_{\Pc_A*(-\Pc_A)})\\
&\simeq&
( K[x_2,\dots,x_n,y_2,\dots,y_n]/\left<x_2y_2,\dots,x_ny_n,{\rm in}_<(g_1), \dots, {\rm in}_<(g_s)\right>)[x_1,y_1].
\end{eqnarray*}
Then $h(K[x_1,y_1]/\left<x_1 y_1\right>, t) = t+1$, and hence
\begin{eqnarray*}
 & & h(K[x_1,\dots,x_n,y_1,\dots,y_n,z]/{\rm in}_<(I_{A^\pm}), t)\\
&=& (t+1 ) h(K[x_1,\dots,x_n,y_1,\dots,y_n]/{\rm in}_{<'} (I_{\Pc_A*(-\Pc_A)}),t)
.
\end{eqnarray*}
Since $\Pc_A*(-\Pc_A)$ is IDP,  it follows from Lemma \ref{lem:initialh} that
\begin{eqnarray*}
h(K[A^\pm], t) &=& h(K[x_1,\dots,x_n,y_1,\dots,y_n,z]/{\rm in}_<(I_{A^\pm}), t)\\
&=& (t+1 ) h(K[x_1,\dots,x_n,y_1,\dots,y_n]/{\rm in}_{<'} (I_{\Pc_A*(-\Pc_A)}),t)\\
&=& (t+1) h (K[\Pc_A*(-\Pc_A)], t)\\
&=& (t+1) h^*(\Pc_A*(-\Pc_A), t).
\end{eqnarray*}
From Proposition~\ref{OHresults}, $h(K[A^\pm], t) $ is palindromic and of degree $d$,
and hence $h^*(\Pc_A*(-\Pc_A), t)$ is palindromic and of degree $d-1$.
Since the dimension of  $\Pc_A*(-\Pc_A)$ is $d$,
 $\Pc_A*(-\Pc_A)$ is Gorenstein of index $2$.
\end{proof}

From the proof of Theorems \ref{main1} and \ref{main2} we obtain the following Corollary.
\begin{Corollary}
Let $A=(\ab_1,\ldots,\ab_n) \in \ZZ^{d \times n}$ be a unimodular configuration and set $A_0=(A, {\bf 0}) \in \ZZ^{d \times (n+1)}$.
	Assume that $\Pc_{A_0} \cap \ZZ^d=\{\ab_1,\ldots,\ab_n,{\bf 0}\}$, and both $\Pc_A$ and $\Pc_{A_0}$ are spanning. 
	Then one has
	\[
	h^*(\Pc_{A_0}*(-\Pc_{A_0}),t)=(1+t)h^*(\Pc_{A}*(-\Pc_{A}),t).
	\]

\end{Corollary}

\section{Unimodular configurations arising from finite simple graphs}

\label{epsection}

Let $G$ be a finite connected simple graph on the vertex set $[d]$
with the edge set $E(G)=\{e_1,\dots,e_n\}$.
Here a graph is called {\em simple} if $G$ has no loops and
no multiple edges.
Given an edge $e = \{i,j\}$ of $G$,
let $\rho(e) = \eb_i + \eb_j \in \RR^d$.
The {\em edge polytope} of $G$ is $\Pc_{A_G}$,
where $A_G$ is the configuration defined by
\[
A_G=( \rho(e_1),\dots,\rho(e_n) )\in \ZZ^{d \times n}.
\]
Since $\Pc_{A_G}$ is a $(0,1)$-polytope, we have
$\Pc_{A_G} \cap \ZZ^d = \{\rho(e_1),\dots,\rho(e_n)\} $.
It is known by \cite[Proposition~1.3]{OH:normal} that
\[
\dim \Pc_{A_G} = 
\left\{
\begin{array}{cc}
d-2 & \mbox{ if } G \mbox{ is bipartite,}\\
d-1 & \mbox{ otherwise.}
\end{array}
\right.
\]
A classification of the graphs $G$ such that $A_G$ is unimodular is as follows:

\begin{Proposition}[{\cite[Theorem~3.3]{OHcentrally}}]
Let $G$ be a finite connected simple graph.
Then the following conditions are equivalent{\rm :}
\begin{itemize}
\item[{\rm (i)}]
The toric ring $K[A_G^\pm]$ is normal{\rm ;}
\item[{\rm (ii)}]
The toric ideal $I_{A_G^\pm}$ has a squarefree initial ideal{\rm ;}
\item[{\rm (iii)}]
The configuration $A_G$ is unimodular (by deleting a redundant row if $G$ is bipartite){\rm ;}
\item[{\rm (iv)}]
All pairs of odd cycles in $G$ have a common vertex.
\end{itemize}  
\end{Proposition}

It follows from {\cite{HMT:edge_linear}} that  $\Pc_{A_G}$ is always spanning.
\begin{Lemma}[{\cite[Proof of Corollary 3.4]{HMT:edge_linear}}]
Let $G$ be a finite connected graph.
Then $\Pc_{A_G}$ is spanning.
\end{Lemma}

Hence we obtain the following.

\begin{Theorem}
\label{edge1}
Let $G$ be a finite connected simple graph on $[d]$. 
Assume that all pairs of odd cycles in $G$ have a common vertex.
	Then 
\begin{enumerate}
		\item[{\rm (1)}] $\Pc_{A_G}*(-\Pc_{A_G})$ is Gorenstein of index $2$ with a regular unimodular triangulation{\rm ;}
		\item[{\rm (2)}] $\Pc_{A_G}+(-\Pc_{A_G})$ is reflexive with a regular unimodular triangulation. In particular, $(\Pc_{A_G}-\ab)+(-\Pc_{A_G}+\ab)$ is a nef-partition, where $\ab$ is an arbitrary lattice point in $\Pc_{A_G}${\rm ;}
		\item[{\rm (3)}] One has $(\Pc_{A_G} \cap \ZZ^d)+(-\Pc_{A_G} \cap \ZZ^d)=(\Pc_{A_G}+(-\Pc_{A_G})) \cap \ZZ^d$.
	\end{enumerate}
\end{Theorem}

On the other hand, in general, $\Pc_{(A_G)_0}$ is not always spanning even if $A_G$ is unimodular.

\begin{Example}
Let $G$ be a cycle of length $3$, i.e., 
$E(G) = \{\{1,2\},\{ 2,3\},\{1,3\} \}$.
Then 
\[
A_G =
\left(
\begin{array}{ccc}
1 & 0 & 1\\
1 & 1 & 0\\
0 & 1 & 1
\end{array}
\right)
\]
is unimodular.
On the other hand, $\Pc_{(A_G)_0}$ is not spanning.
\end{Example}

We determine when $\Pc_{(A_G)_0}$ is spanning.

\begin{Lemma}
	Let $G$ be a finite connected simple graph. 
    Then $\Pc_{(A_G)_0}$ is spanning if and only if $G$ is bipartite.
\end{Lemma}
\begin{proof}
Note that a full-dimensional lattice polytope $\Pc \subset \RR^d$ with ${\bf 0} \in \Pc$ is spanning if and only if it holds $\sum_{\ab \in \Pc \cap \ZZ^d} \ZZ \ab = \ZZ^d$.

We assume that $G$ is not bipartite.
Then $\Pc_{(A_G)_0}$ is full-dimensional.
Since for any lattice point $(a_1,\ldots, a_d) \in \Pc_{(A_G)_0} \cap \ZZ^d$, $a_1+\cdots+a_d$ is even, it follows that for any lattice point $(x_1,\ldots, x_d) \in \sum_{\ab \in \Pc_{(A_G)_0} \cap \ZZ^d} \ZZ \ab$, $x_1+\cdots+x_d$ is even.
Hence $\Pc_{(A_G)_0}$ is not spanning.

Conversely, we assume that $G$ is bipartite. 
Let $A$ be a configuration obtained by deleting a row of $A_G$.
Then $\Pc_{(A_G)_0}$ is unimodularly equivalent to a full-dimensional lattice polytope $\Pc_{A_0} \subset \RR^{d-1}$.
It is known by \cite[Example 1 (p.273)]{integer} that any nonzero maximal minor of $A_0$ is $\pm 1$.
Thus $\sum_{\ab \in \Pc_{A_0} \cap \ZZ^{d-1}} \ZZ \ab = \ZZ^{d-1}$,
and hence $\Pc_{A_0} $ is spanning.
\end{proof}

Therefore, we have obtain the following.

\begin{Theorem} 
\label{edge2}
Let $G$ be a finite connected simple bipartite graph on $[d]$.
Then
		\begin{enumerate}
		\item[{\rm (1)}] $\Pc_{(A_G)_0}*(-\Pc_{(A_G)_0})$ is Gorenstein of index $2$ with a regular unimodular triangulation{\rm ;}
		\item[{\rm (2)}] $\Pc_{(A_G)_0}+(-\Pc_{(A_G)_0})$ is reflexive with a regular unimodular triangulation. In particular, $\Pc_{(A_G)_0}+(-\Pc_{(A_G)_0})$ is a nef-partition{\rm ;}
		\item[{\rm (3)}] One has $(\Pc_{(A_G)_0} \cap \ZZ^d)+(-\Pc_{(A_G)_0} \cap \ZZ^d)=(\Pc_{(A_G)_0}+(-\Pc_{(A_G)_0})) \cap \ZZ^d$;
		\item[{\rm (4)}] One has $h^*(\Pc_{(A_G)_0}*(-\Pc_{(A_G)_0}),t)=(1+t)h^*(\Pc_{A_G}*(-\Pc_{A_G}),t)$.
	\end{enumerate}
\end{Theorem}

\begin{acknowledgement}
The authors are grateful to an anonymous referee for his careful reading of the manuscript and for his useful comments. 
The authors were partially supported by JSPS KAKENHI 18H01134, 19K14505 and 19J00312.
\end{acknowledgement}
\bibliographystyle{plain}
\bibliography{Bibliography.bib}

\end{document}